\documentclass{amsart}
\usepackage{graphicx} 
\usepackage{amsmath,amssymb,amsthm}
\usepackage{hyperref}
\usepackage[shortlabels]{enumitem}
\usepackage{comment}

\newcommand{\Z}{\mathbb{Z}}

\theoremstyle{plain}
\newtheorem{thm}{Theorem}
\newtheorem{lem}{Lemma}[section]
\newtheorem{prop}[lem]{Proposition}

\newtheorem{rem}[lem]{Remark}
\newtheorem{ex}[lem]{Example}

\theoremstyle{definition}
\newtheorem{defn}[lem]{Definition}

\title{On upper bounds of frieze patterns}
\author{Jon Cheah}
\address{
Department of Mathematics   \\
The University of Hong Kong \\
Pokfulam Road               \\
Hong Kong}
\email{joncheah@connect.hku.hk}
\author{Antoine de Saint Germain}
\address{
Department of Mathematics and New Cornerstone Science Laboratory   \\
The University of Hong Kong \\
Pokfulam Road               \\
Hong Kong}
\email{adsg96@hku.hk}

\usepackage{tikz}
\usetikzlibrary{shapes.geometric}

\begin{document}
\begin{abstract}
    In this note, we show that the sequence of maximum values in frieze patterns of type $A_n$ is the sequence of Fibonacci numbers, and that of frieze patterns of type $C_n$ is the sequence of odd Fibonacci numbers.
\end{abstract}

\maketitle
\tableofcontents
\addtocontents{toc}{\protect\setcounter{tocdepth}{1}}

\section{Introduction}
    For a positive integer $n$, a (Coxeter) \emph{frieze pattern} of width $n$ is an infinite array of $(n+2)$ staggered rows 

    \begin{equation}\label{eq:Coxeter}
        \begin{array}{ccccccccccccccc}
            \cdots &   & 1 &   & 1 &   & 1 &     & \cdots &   &   &   &   &   &\\
              & \cdots &   & a_1 &   & b_1 &   & c_1 &     & \cdots &   &   &   &   &\\
              &   & \cdots &   & a_2 &   & b_2 &   & c_2 &     & \cdots &   &   &   &\\
              &   &   &   &   & \ddots &   & \ddots &   & \ddots &      &   &   &   &\\
              &   &   &   & \cdots &   & a_n &   & b_n &   & c_n &    & \cdots &   &\\
              &   &   &   &   & \cdots &   & 1 &   & 1 &   & 1 &     & \cdots &\\
        \end{array}
    \end{equation}
    consisting entirely of positive integers satisfying the so-called {\it diamond rule}: $ad = 1+bc$ whenever the positive integers $a$, $b$, $c$, and $d$ form a ``diamond'' 
    \begin{equation*}
        \begin{array}{ccc}
              & b &   \\
            a &   & d \\
              & c & 
        \end{array}
    \end{equation*}
    in the array. Frieze patterns were introduced by H. S. M. Coxeter \cite{Coxeter1971}, and further studied together with J.H. Conway \cite{CC73}. In the years following their inception, many modifications of the diamond rule were considered (see \cite[\S 5]{Sophie-M:survey} for a large list). One important generalisation was given in the mid-2000s, when P. Caldero and F. Chapoton \cite{CC06} realised frieze patterns of a given width n 
    as evaluations of distinguished ring homomorphisms  
    on all cluster variables in the cluster algebra (with trivial coefficients) of type $A_n$, thus unveiling the cluster algebra nature of frieze patterns. Motivated by this, I. Assem, C. Reutenauer and D. Smith \cite{ARS:friezes}
    introduced the following generalisation of Coxeter's frieze patterns.

    \medskip
    \begin{defn}\label{def:frieze}
        Let $A = (a_{i,j})$ be an $n\times n$ symmetrisable generalised Cartan matrix (see Definition \ref{def:cartan}). A {\it frieze pattern of type A} is a map 
    $f : \{1,\ldots , n\} \times \mathbb{Z} \longrightarrow \mathbb{Z}_{>0}$ such that
    \begin{equation*}
        f(i,m) f(i,m+1) = 1 + \prod_{j=i+1}^n f(j,m)^{-a_{j,i}} \prod_{j =  1}^{i-1} f(j,m+1)^{-a_{j,i}}, \quad \forall i, \forall m.
    \end{equation*}
    The set of frieze patterns of type A is denoted ${\rm Frieze}(A)$.
    \end{defn}
    \medskip

    When $A=(a_{i,j})$ is the so-called standard Cartan matrix of type $A_n$ (see \S \ref{ss:Cartan}), one recovers (Coxeter) frieze patterns of width $n$ by placing 
    each tuple $(f(1,m), \ldots, f(n,m)), m \in \mathbb{Z}$ in the $m^{\rm th}$ North-West to South-East diagonal in \eqref{eq:Coxeter}.  

    Remarkably, the set ${\rm Frieze}(A)$ is finite if and only if $A$ is of ``finite type", i.e. $A$ is positive definite (c.f. \cite{FZ:II,GM:finite}). In other words, frieze patterns in the sense of Definition \ref{def:frieze} exhibit a classification similar to the Cartan-Killing classification of semisimple Lie algebras over $\mathbb{C}$, or to the Fomin-Zelevinsky classification \cite{FZ:II} of cluster algebras! It is well known that such a classification can be broken down into 4 infinite families (denoted $A_n$, $B_n$, $C_n$ and $D_n$) and five exceptional types $G_2,F_4,E_6, E_7$ and $E_8$. 

    When $A$ is of finite type, it is natural to ask for the cardinality $c(A)$ of the set ${\rm Frieze}(A)$. A first step in this direction was achieved by Coxeter and Conway \cite{CC73}, who constructed a beautiful bijection from ${\rm Frieze}(A_n)$ to the set of triangulations of the regular $(n+3)$-gon, whereby $c(A_n) = C_{n+1} = \frac{1}{n+2} \binom{2n+2}{n+1}$, {\it the $(n+1)^{\rm st}$ Catalan number}. Subsequently, formulas 
    have been obtained for $c(A)$ when $A$ is of type $B_n, C_n, D_n,G_2, F_4$ and $E_6$. At the time of writing,  
    $c(E_7)$ and $c(E_8)$ are not known (but see \cite{FP:Dn,GM:finite} for conjectural values). 

    Keeping with the assumption that $A$ is of finite type, Coxeter \cite{Coxeter1971} proved in type $A_n$ and \cite{Keller:periodicity} proved in general that every $f \in {\rm Frieze}(A)$ is periodic in the second component, i.e. there exists $N_A$ such that 
    \[
        f(i,m) = f(i,m+N_A), \quad \forall i, \forall m.
    \]

    In light of the finiteness of 
    ${\rm Frieze}(A)$ and the periodicity of each $f \in {\rm Frieze}(A)$, one has the well-defined 
    number 
    \begin{equation}
        u_A := \max (f (i,m) : f \in {\rm Frieze}(A), i \in \{1,\ldots , n\}, m \in \mathbb{Z}).
    \end{equation}
    In this article, we determine $u_A$ when $A$ is of type $A_n$ and $C_n$. Our results are summarised as follows. 

    \medskip
    \begin{thm}\label{mainthm} Let $A$ be a Cartan matrix of finite type. Then, 
    \begin{enumerate}[\normalfont(1)]
        \item if $A$ is of type $A_n$, then $u_A = F_{n+2}$; and
        \item if $A$ is of type $C_n$ then $u_A = F_{2n+1}$, 
    \end{enumerate}
    where $F_n$ is the $n^{\rm th}$ Fibonacci number, defined by $F_0 = 0$, $F_1 = 1$ and $F_{k+2} = F_{k+1} + F_k$. 
    \end{thm}

    \medskip

    We point out that G. Muller \cite{muller:finite} has recently obtained a (very large but explicit) upper bound for $u_A$, using only elementary properties of Cartan matrices. 
    
    Statement (1) of Theorem 1 and its proof have been formally verified by the Lean4 proof assistant. The formalisation project is available at \url{https://antoine-dsg.github.io/frieze_patterns}.

    The article is structured as follows. In \S \ref{ss:Cartan}, we recall the definitions of Cartan matrices and folding. We prove statement (1) of Theorem \ref{mainthm} in \S \ref{ss:proof1} and statement (2) of Theorem \ref{mainthm} in \S \ref{ss:proof2}.

\section{Proof of Theorem \ref{mainthm}}\label{s:proof1}
Throughout this section, we let $n$ be a fixed positive integer. 
\subsection{Cartan matrices and folding}\label{ss:Cartan}
\begin{defn}\label{def:cartan}
    A {\it symmetrisable generalised Cartan matrix} is an $n \times n$ integral matrix 
$A = (a_{i, j})_{i, j = 1, \ldots, n}$ such that 

\begin{enumerate}
    \item $a_{i ,i} = 2$ for all $i \in \{1, \ldots, n\}$,
    \item $a_{i, j} \leq 0$ for all $i, j \in \{1, \ldots, n\}$ such that $i \neq j$, with $a_{i,j} = 0$ if and only if $a_{j,i}=0$, and
    \item there exists a diagonal matrix $D$ with positive integers on the diagonal such that $DA$ is symmetric.
\end{enumerate}
A symmetrisable generalised Cartan matrix is said to be of {\it finite type} if it is positive definite. If there exists a simultaneous permutation of rows and columns of $A$ such that $A$ is block diagonal, we say that $A$ is {\it decomposable}. If not, we say that $A$ is {\it indecomposable}. 
\end{defn}
It is well-known that indecomposable Cartan matrices of finite type are classified into four classical series $A_n, B_n, C_n, D_n$ and five exceptional types labeled as $G_2, F_4, E_6, E_7, E_8$. For example, $G_2 = \left(\begin{array}{cc} 2 & -3 \\ -1 & 2\end{array}\right)$. In this article, the {\it standard Cartan matrix of type $A_n$} is the $n\times n$ integer matrix
\begin{equation}\label{eq:Cartan-An}
A_n = \left(\begin{array}{cccccc} 2 & -1 & 0 & \cdots & 0 & 0\\
-1 & 2 & -1 & \cdots & 0 & 0 \\
0 & -1 & 2 & \cdots & 0 & 0\\
\vdots & \vdots & \vdots& \ddots & \vdots & \vdots\\
0 & 0 & 0 & \cdots & 2 & -1\\
0 & 0 & 0& \cdots & -1 & 2\end{array}\right).
\end{equation}
Similarly, the {\it standard Cartan matrix of type $C_n$} is the $n\times n$ integer matrix
\begin{equation}\label{eq:Cartan-Cn}
     C_n = \left(\begin{array}{cccccc} 2 & -1 & 0 & \cdots & 0 & 0\\
-1 & 2 & -1 & \cdots & 0 & 0 \\
0 & -1 & 2 & \cdots & 0 & 0\\
\vdots & \vdots & \vdots& \ddots & \vdots & \vdots\\
0 & 0 & 0 & \cdots & 2 & -2\\
0 & 0 & 0& \cdots & -1 & 2\end{array}\right).
\end{equation}

It is well-known (see \cite{FP:Dn}) that 
$C_n$ is a ``folding" of type $A_{2n-1}$ via a certain automorphism $\sigma$. This allows us to reduce questions about $C_n$ to 
``$\sigma$-invariant" questions about $A_{2n-1}$. The precise meaning of ``$\sigma$-invariant" depends on the question; a precise formulation of ``$\sigma$-invariant" frieze patterns will be given in \S \ref{ss:proof2}.

\subsection{The case of type $A_n$}\label{ss:proof1}
    The following definition is a special case of Definition \ref{def:frieze} in type $A_n$.  
\begin{defn}\label{def:friezeAn}
    A \emph{frieze pattern of type $A_n$} is a map $f:\{1,\dots,n\}\times \Z \to \Z_{>0}$ such that for all $i\in \{1,\dots,n\}$ and all $m\in \Z$, we have
    \begin{equation}\label{eqn:diamond}
        f(i,m)f(i,m+1) = 1 + f(i+1,m)f(i-1,m+1),
    \end{equation}
    where by convention $f(0,m)=f(n+1,m):=1$ for all $m\in \Z$.
\end{defn}

In this subsection, a frieze pattern will always mean a frieze pattern of type $A_n$. 

\begin{ex}\label{ex:friezeA5}
    The following array describes a frieze pattern of type $A_5$.
    \[
\begin{matrix}
    1 && 1 && 1 && 1 && 1 && 1 &&   \\
    & 2 && 3 && 3 && 1 && 2 && 3 &&  \\
    && 5 && 8 && 2 && 1 && 5 && 8 &&  && \\
    &\cdots && 13 && 5 && 1 && 2 && 13 && 5 && \cdots  \\
    &&&& 8 && 2 && 1 && 5 && 8 && 2 &&   \\
    &&&&& 3 && 1 && 2 && 3 && 3 && 1 &&   \\
    &&&&&&1 && 1 && 1 && 1 && 1 && 1 && 
\end{matrix}
\]    
\end{ex}
    For each $m\in \Z$, the $n$-tuple $(f(1,m),\dots, f(n,m))$ is called the $m$th \emph{diagonal} of $f$. Conversely, an $n$-tuple of positive integers $(a_1, \ldots , a_n)$ is called a {\it diagonal of a frieze pattern} if there exists a frieze pattern $f$ and an $m \in \mathbb{Z}$ such that $a_i = f (i,m)$ for all $i \in \{1,\ldots, n\}$. By recursively solving (\ref{eqn:diamond}) with respect to the (lexicographical) total order on $\{1,\dots,n\}\times \Z$, a frieze pattern is uniquely determined by any one of its diagonals. The following criterion is well-known.

\begin{lem}[\protect{\cite[(24)]{CC73}}]\label{friezeFlute}
    An $n$-tuple $(a_1, a_2,\dots,a_n)$ of positive integers is a diagonal of a frieze pattern if and only if 
    \begin{equation}\label{cond:flute}\tag{$\ast$}
        \text{for all } i \in \{1,\dots,n\},\quad a_i \text{ divides } a_{i-1} + a_{i+1},
    \end{equation}
    where by convention $a_0=a_{n+1}=1$.
\end{lem}

In light of Lemma \ref{friezeFlute}, proving statement (1) of Theorem \ref{mainthm} reduces to proving the following proposition.

\begin{prop}\label{boundDiag}
    Let $n$ be a positive integer. The following statements hold.
    \begin{enumerate}[\normalfont(1)]
        \item If $(a_1,\ldots, a_n)$ satisfies  {\normalfont{(\ref{cond:flute})}}, then $a_i \leq F_{n+2}$ for all $i \in \{1,\ldots, n\}$.
        
        \item If $n$ is even, the $n$-tuple 
    $$\left(F_4, F_6,\dots, F_{n},  F_{n+2}, F_{n+1}, F_{n-1},\dots, F_5, F_3\right)$$
    satisfies {\normalfont{(\ref{cond:flute})}}.
    
        \item If $n$ is odd, the $n$-tuple 
     $$\left(F_4, F_6,\dots, F_{n+1},  F_{n+2}, F_{n}, F_{n-2},\dots, F_5, F_3\right)$$
     satisfies {\normalfont{(\ref{cond:flute})}}.
    \end{enumerate}
\end{prop}

In order to prove Proposition \ref{boundDiag}, we need the following lemma.

\begin{lem}\label{lem:reductioncases}
    If an $n$-tuple $(a_1, a_2,\dots,a_n)$ of positive integers satisfies {\normalfont{(\ref{cond:flute})}}, and $a_n\neq1$, then there exists $i\in\{1,\dots,n\}$ such that $a_i=a_{i-1}+a_{i+1}$.
\end{lem}
\begin{proof}
    If there is no such $i$, the condition {\normalfont{(\ref{cond:flute})}} gives us that $a_{i-1}+a_{i+1}\geq 2a_i$ for $i=1,\ldots, n$. Thus, we have
    $a_{i+1}-a_i=a_{i+1}+a_{i-1}-a_{i-1}-a_i\geq a_i - a_{i-1}$ for each $i$.
    Combining these inequalities, and recalling that $a_0=1$, we obtain 
    $$a_{n+1} - a_{n} \geq a_{n} -a_{n-1} \geq \cdots \geq a_2 - a_1 \geq a_1 - a_0 \geq 0.$$
    This implies $a_{n+1} - a_{n} = 1 - a_n \geq 0$, so $a_n=1$.
\end{proof}

We write $u_n \stackrel{\text{def}}{=} u_{A_n}$. The following lemma is the base case needed for part (1) of Proposition \ref{boundDiag}. 

\begin{lem}\label{lem:smolfrieze}
    We have $u_1 = 2$ and $u_2=3$.
\end{lem}
\begin{proof}
    When $n=1$, each diagonal of a frieze is a $1$-tuple $(a_1)$ such that $a_1$ divides $a_0 + a_2 = 1+1=2$. Hence $a_1= 1$ or $2$ and $u_1 \leq 2$. On the other hand, Figure \ref{fig:frieze-width1} implies that $u_1 \geq 2$, and so $u_1 =2$. 

    When $n=2$, each diagonal of a frieze is a pair of integers $(a_1,a_2)$ satisfying condition (\ref{cond:flute}). Applying Lemma \ref{lem:reductioncases}, if $a_2=1$, then $a_1$ is $1$ or $2$. When $a_2\neq1$, either $a_1 = 1+a_2$ or $a_2 = 1+a_1$. In the case of the former, $a_2$ divides $1+a_1= 2+a_2$ implies $a_2=2$ and $a_1=3$. A symmetric argument yields $a_1=2$ and $a_2=3$ in the latter case. Thus $u_2 \leq 3$. To conclude, note that Figure \ref{fig:frieze-width2} is a frieze pattern of width $2$ containing the number $3$. 
\end{proof}

\begin{figure}[ht]
    \centering
    \[
\begin{matrix}
    &1 && 1 && 1 && 1 && 1 &&  \\
    \cdots&& 1 && 2 && 1 && 2 && 1 &&\cdots \\
    &&& 1 && 1 && 1 && 1 && 1 
\end{matrix}
\]
    \caption{A frieze pattern of type $A_1$.}
    \label{fig:frieze-width1}
\end{figure}

\begin{figure}[ht]
    \centering
    \[
\begin{matrix}
    1 && 1 && 1 && 1 && 1 && 1 \\
    & 1 && 3 && 1 && 2 && 2 && 1  \\
   \cdots && 2 && 2 && 1 && 3 && 1 && 2 && \cdots  \\
    &&&1 && 1 && 1 && 1 && 1 && 1
\end{matrix}
\]
    \caption{A frieze pattern of type $A_2$.}
    \label{fig:frieze-width2}
\end{figure}

\begin{proof}[Proof of Proposition \ref{boundDiag}.]
    We prove (1) by induction on $n$. The base cases $n=1$ and $n=2$ are proved in Lemma \ref{lem:smolfrieze}.
    Fix $n \geq 3$, and 
    suppose our claim holds for all $m < n$. Consider an 
    arbitrary $n$-tuple $(a_1,\ldots, a_n)$ satisfying {\normalfont{(\ref{cond:flute})}}. If $a_n = 1$, one sees that $(a_1,\ldots, a_{n-1})$ is an $(n-1)$-tuple satisfying {\normalfont{(\ref{cond:flute})}}, so by the inductive hypothesis, $a_i \leq F_{n+1} \leq F_{n+2}$ for all $i \in \{1,\ldots, n\}$, and we are done. If $a_n \neq 1$, there exists by 
    Lemma \ref{lem:reductioncases} an index $i \in \{1,\ldots , n\}$ such that $a_i = a_{i-1}+a_{i+1}$. In this case, $a_{i-1}$ divides $a_{i-2}+a_i$ and $a_i = a_{i-1} + a_{i+1}$, so $a_{i-1}$ divides $a_{i-2}+a_{i+1}$. Similarly, $a_{i+1}$ divides $a_{i-1}+a_{i+2}$. In other words the $(n-1)$-tuple 
    \begin{equation}\label{eq:(n-1)diag}
    (a_1,\ldots, a_{i-1},\widehat{a}_i, a_{i+1}, \ldots, a_n)
    \end{equation}
    satisfies {\normalfont{(\ref{cond:flute})}} --- here $\widehat{\cdot}$ denotes omission. By induction, $a_k \leq F_{n+1} \leq F_{n+2}$ for all $k \in \{1,\ldots , n\} \backslash \{i\}$. It now remains to show that $a_i \leq F_{n+2}$. Applying the above case analysis to the $(n-1)$-tuple given in \eqref{eq:(n-1)diag}, we  conclude that there exists an index $j \in \{1,\ldots, n\} \backslash \{i\}$ such that the 
    $(n-2)$-tuple 
    \[
    (a_1,\ldots,\widehat{a}_i, \ldots, \widehat{a}_j,\ldots,  a_n)
    \]
    satisfies {\normalfont{(\ref{cond:flute})}}, and so $a_k \leq F_n$ for all $k\in \{1,\dots,n\}\setminus \{i,j\}$. If $j = i+1$ or $i-1$, then $$a_i = a_{i-1} + a_{i+1} \leq F_{n+1} + F_n = F_{n+2}.$$ Otherwise, $a_i \leq 2 F_n \leq F_{n+2}$. In either case, $a_i \leq F_{n+2}$, completing the proof of (1).

    The proofs of (2) and (3) are a straightforward 
    computation. The key is to notice that away from the middle of the tuple, we have that
    $$F_{i-2} + F_{i+2} = F_{i-2} + F_{i} + F_{i+1} = F_{i-2} + 2F_{i} + F_{i-1} = 3F_i$$ is divisible by $F_i$. Near the middle of the tuple, the condition (\ref{cond:flute}) is immediate.
\end{proof}

\begin{rem}\label{eq:zigzag}
    The diagonals constructed in 2) and 3) of Proposition \ref{boundDiag} give rise, via Lemma \ref{friezeFlute}, to frieze patterns 
    ``corresponding to the zig-zag triangulation of the $(n+3)$-gon'' in the terminology of \cite{CC73}, also called the ``snake" triangulation in \cite[figure 13]{FZ:II}. A zig-zag triangulation of the octogon is given below.

\begin{center}
\begin{tikzpicture}
    \node[draw, minimum size=2cm, regular polygon, regular polygon sides=8] (octagon) {};
    \draw (octagon.corner 1) -- (octagon.corner 3);
    \draw (octagon.corner 3) -- (octagon.corner 8);
    \draw (octagon.corner 4) -- (octagon.corner 8);
    \draw (octagon.corner 4) -- (octagon.corner 7);
    \draw (octagon.corner 7) -- (octagon.corner 5);
\end{tikzpicture}
\end{center}
\end{rem}

\subsection{The case of type $C_n$}\label{ss:proof2}
We rewrite Definition \ref{def:frieze} in the case of type $C_n$.  
\begin{defn}\label{def:friezeCn}
    A \emph{frieze pattern of type $C_n$} is a map $f:\{1,\dots,n\}\times \Z \to \Z_{>0}$ such that for all $m\in \Z$, we have 
    \begin{equation}\label{eqn:Cn-rule}
    f(i,m)f(i,m+1) = \begin{cases}
        1 + f(i+1,m)f(i-1,m+1) & \quad \text{ if } i = 1, \ldots , n-1, \\
        1 + (f(n-1,m+1))^2 & \quad \text{ if } i =n,
    \end{cases}
    \end{equation}
    where by convention $f(0,m):=1$ for all $m\in \Z$.
\end{defn}
The following lemma is a phenomenon of ``folding" from friezes of type $A_{2n-1}$ to friezes of type $C_n$; a geometric version can be found in \cite[Theorem 4.2]{FP:Dn}.
\begin{lem}\label{l:foldingAnCn}
Let $a \in {\rm Frieze}(A_{2n-1})$, and suppose that for all $m \in \Z$,
\begin{equation}\label{eq:an-to-cn1}
    a(n+j,m) = a(n-j,m+j), \quad j \in \{1,\ldots ,n-1\}.
\end{equation}
Then the function $c : \{1,\ldots, n\} \times \Z \longrightarrow \Z_{>0}$ given by 
\begin{equation}\label{eq:an-to-cn2}
    c(n-i,m) = a(n-i,-m+i), \quad i \in \{0,\ldots, n-1\}, m \in \Z,
\end{equation}
is a frieze pattern of type $C_n$. Moreover, every frieze pattern of type $C_n$ arises in this way.
\end{lem} 
\begin{proof}
    To prove the first claim, it is sufficient to show that the function $c$ thus defined satisfies \eqref{eqn:Cn-rule} for all indices. Let $m \in \Z$ be arbitrary. We proceed by a case-by-case check on the index $i$. 

    {\bf Case 1:} Suppose that $i \in \{1,\ldots, n-1\}$. Set $k = n-i$, so that $k \in \{1,\ldots, n-1\}$.  
    \begin{align*}
    c(k,m) \; c(k,m+1) &= a(k,-m+i)\;  a(k, -(m+1)+i) \\
    &= 1 + a(k+1,-(m+1)+i) \; a(k-1,-m+i) \\
    &= 1 + a(n-(i-1),-m+(i-1)) \, a (n-(i+1),-(m+1)+(i+1)) \\
    & = 1 + c (k+1,m)\; c(k-1,m+1).
    \end{align*}

    {\bf Case 2:} Suppose that $i = 0$. Then 
    \begin{align*}
        c (n,m) \, c(n,m+1) &= a(n,-m) \, a (n, -m-1) \\
        &= 1 + a (n+1,-m-1) \, a(n-1,-m) \\
        &= 1 + a(n+1,-m-1) \, a (n-1,-m-1 +1) \\
        &= 1 + a(n-1,-(m+1)+1)^2, \\
        &= 1 + c(n-1,m+1)^2,
    \end{align*}
    where in the penultimate line we have used the assumption  in the statement of the Lemma. This concludes the proof that every $A_{2n-1}$ frieze pattern gives rise to a $C_n$ frieze. To see that the reverse hold, let $c \in {\rm Frieze}(C_n)$, and define 
    \[
        a : \{1,\ldots, 2n-1\} \times \Z \longrightarrow \Z_{>0},
    \]
    via the conditions \eqref{eq:an-to-cn1} and \eqref{eq:an-to-cn2}. To show that the resulting function satisfies \eqref{eqn:diamond}, we again proceed by cases on $i \in \{1,\ldots, 2n-1\}$. 
    
    {\bf Case 1:} $i \in \{1,\ldots, n-1\}$. Set $k = n-i \in \{1,\ldots, n-1\}$. Then,
    \begin{align*}
        a (i,m) \, a(i,m+1) &= a (n-k,m) \, a(n-k,m+1) \\
        &=  a (n-k,-(-m+k)+k) \, a(n-k,-(-m-1+k)+k) \\
        &=  c (n-k,-m+k) \, c(n-k,-m+k-1)\\
        &= 1 + c(n-k+1,-m+k-1) \, c(n-k-1,-m+k)\\
        &= 1 + c (n - (k-1), -m+k-1) \, c(n - (k+1),-m+k) \\
        &= 1 + a (i+1,m) \, a (i-1,m+1).
    \end{align*}
    
    {\bf Case 2:} $i = n$. Then,
    \begin{align*}
        a (n,m) \, a(n,m+1) &= c (n,-m) \, c(n,-m-1) \\
        &= 1 + c(n-1,-m)^2 \\
        &= 1 + a (n-1,m+1)^2 \\
        &= 1 + a (n+1,m)\, a (n-1,m+1).
    \end{align*}

    The remaining case $i \in \{n+1,\ldots, 2n-1\}$ reduces to the first case via \eqref{eq:an-to-cn1}, and we omit the details.
\end{proof}
\begin{ex}\label{ex:foldedC3}
The following is the frieze pattern of type $C_3$ obtained from the frieze of type $A_5$ given in Example \ref{ex:friezeA5} via Lemma \ref{l:foldingAnCn}. 
    \[
\begin{matrix}
    1 && 1 && 1 && 1 && 1 && 1 && 1  \\
    & 3 && 3 && 2 && 1 && 3 && 3  && 2 \\
   \cdots && 8 && 5 && 1 && 2 && 8 && 5 && 1 && \cdots  \\
    &&& 13 && 2 && 1 && 5 && 13 && 2 && 1  \\
    &&&&1 && 1 && 1 && 1 && 1 && 1 && 1
\end{matrix}
\]

\end{ex}

We now describe a more geometric incarnation of Lemma \ref{l:foldingAnCn}, as given in \cite{FP:Dn}. To do so, we recall the well-known realisation of frieze patterns of type $A_n$ in terms of polygons (c.f. \cite[\S 4.2]{Sophie-M:survey} or \cite[\S 12.2]{FZ:II}). 

Consider the regular polygon $P_{n+3}$ with $(n+3)$ vertices (labelled cyclically by $\{1,\ldots, n+3\}$) and $(n+3)$ edges. We call a {\it diagonal} any pair $\{i,j\}$ of non-consecutive vertices. Each diagonal can be uniquely written as an ordered pair $(i,j)$ with $i < j$. We say that two diagonals $(i_1,i_2)$ and $(j_1,j_2)$ {\it intersect} if $i_1= j_1$ and $i_2 = j_2$, or if
\[
    i_1<j_1<i_2<j_2, \text{ or } j_1<i_1<j_2<i_2.
\]

Otherwise, the two diagonals are said to be {\it non-intersecting}. A {\it triangulation} of $P_{n+3}$ is any maximal (w.r.t. inclusion) collection of (pairwise) non-intersecting diagonals. Each triangulation contains exactly $n$ diagonals. By \cite{CC73}, there is a one-to-one correspondence between frieze patterns of type $A_n$ and triangulations of $P_{n+3}$ (see Remark \ref{eq:zigzag} for an example of a triangulation).

According to the proof of \cite[Theorem 4.2]{FP:Dn}, frieze patterns of type $C_n$ are in one-to-one correspondence (using the algebraic procedure from Lemma \ref{l:foldingAnCn}) with frieze patterns of type $A_{2n-1}$ coming from triangulations of $P_{2n+2}$ that are symmetric under central symmetry, i.e. the symmetry which sends the vertex $i$ to the vertex $i + n+1$. 

{\it Proof of (2) of Theorem \ref{mainthm}.}
By Lemma \ref{l:foldingAnCn} and statement (1) of Theorem \ref{mainthm}, we have $u_{C_n} \leq u_{A_{2n-1}} = F_{2n+1}$. It remains to show that there exists a frieze pattern of type $C_n$ that realises this upper bound. By Remark \ref{eq:zigzag} and the discussion following Lemma \ref{l:foldingAnCn}, it is sufficient to show that the zig-zag triangulation of $P_{2n+2}$ is symmetric under the symmetry sending the vertex $i$ to the vertex $i + n+1$. One checks directly that diagonal of the zig-zag triangulation is of the form $(2+i,2n+2-i)$ for $i = 0, \ldots, n-2$ or $(2+i+1,2n+2-i)$ for $i = 0, \ldots, n-3$, and that the set of diagonals in this triangulation is indeed preserved. 
The resulting frieze pattern of type $C_n$, denoted $c$, satisfies $c(i,0) = F_{2i+1}$ for $i \in \{1,\ldots, n\}$ (see Figure \ref{ex:foldedC3} for the case $n= 3$). \qed

\textit{Acknowledgements.} We thank Aksel Chan, Bock-Man Cheung and Eaton Liu for their  contributions to formalising this result in the Lean4 proof assistant. The first author is supported by the Hong Kong PhD Fellowship Scheme (HKPFS) and the HKU Presidential PhD Scholar Programme (HKU-PS). The second author was partially supported by the New Cornerstone Science Foundation through the New Cornerstone Investigator Program awarded to Professor Xuhua He.

\bibliographystyle{alpha}
\bibliography{citelist}
\end{document}